\documentclass[11pt,reqno]{amsart}
\usepackage{graphicx}
\usepackage{amsmath, mathtools, amsfonts, amssymb, amsthm}
\usepackage[margin=1in]{geometry}
\usepackage{color}
\usepackage{xcolor}
\usepackage{soul}
\usepackage{tikz}
\usepackage{verbatim}
\usepackage[normalem]{ulem}
\usepackage{hyperref}
\usetikzlibrary{fit,shapes.geometric}

\def\R{\mathbb{R}}

\theoremstyle{plain}

\newtheorem{theorem}{Theorem}[section]
\newtheorem{utheorem}{\textrm{\textbf{Theorem}}}

\newtheorem{corollary}[theorem]{Corollary}

\newtheorem{lemma}[theorem]{Lemma}

\theoremstyle{definition}

\newtheorem{remark}[theorem]{Remark}
\newtheorem{example}[theorem]{Example}

\numberwithin{equation}{section}

\begin{document}
\title{Sufficient conditions for total positivity, compounds, and Dodgson condensation}

\author{Shaun Fallat}
\address[S.~Fallat]{Department of Mathematics and Statistics, University of Regina, Regina, SK, CA}
\email{\tt shaun.fallat@uregina.ca}

\author{Himanshu Gupta}
\address[H.~Gupta]{Department of Mathematics and Statistics, University of Regina, Regina, SK, CA}
\email{\tt himanshu.gupta@uregina.ca}

\author{Charles R. Johnson}
\address[C.~Johnson]{Department of Mathematics, College of William and Mary, Williamsburg, VA, USA}
\email{\tt crjohn@wm.edu}

\date{\today}



\begin{abstract}
   A $n$-by-$n$ matrix is called totally positive ($TP$) if all its minors are positive and $TP_k$ if all of its $k$-by-$k$ submatrices are $TP$. For an arbitrary totally positive matrix or $TP_k$ matrix, we investigate if the $r$th compound ($1<r<n$) is in turn $TP$ or $TP_k$, and demonstrate a strong negative resolution in general. Focus is then shifted to Dodgson's algorithm for calculating the determinant of a generic matrix, and we analyze whether the associated condensed matrices are possibly totally positive or $TP_k$. We also show that all condensed matrices associated with a $TP$ Hankel matrix are $TP$.
\end{abstract}
\maketitle

\section{Introduction and main results}\label{sec:introduction}
Denote the set of all $m$-by-$n$ matrices with entries from $\R$ by $M_{m,n}(\R)$. We write $M_{n}(\R)$ when $m=n$. For $A \in M_{m,n}(\R)$, $\alpha \subseteq \{1,2,\ldots,m\}$, and $\beta \subseteq \{1,2,\ldots,n\}$, the submatrix of $A$ lying in rows indexed by $\alpha$ and columns indexed by $\beta$ is denoted by $A[\alpha,\beta]$. The complimentary submatrix obtained from $A$ by deleting the rows indexed by $\alpha$ and columns indexed by $\beta$ is denoted by $A(\alpha,\beta)$. If $A \in M_{n}(\R)$ and $\alpha=\beta$, then the \emph{principal submatrix} $A[\alpha,\alpha]$ is abbreviated to $A[\alpha]$, and the complementary principal submatrix to $A(\alpha)$. For a set $\alpha \subseteq \{1,2,\ldots,n\}$, we denote its complement by $\alpha^{c}$, and its cardinality to $|\alpha|$. Thus, $A(\alpha,\beta)= A[\alpha^{c},\beta^{c}]$.

A {\em minor} in a given matrix $A$ is the determinant of a (square) submatrix of $A \in M_{m,n}(\R)$. For example, if $|\alpha| = |\beta|$, the ($\alpha,\beta$)-minor of $A$ will be denoted by ${\rm det}A[\alpha,\beta]$, and the $\alpha$-principal minor is denoted by ${\rm det}A[\alpha]$, in the case when $A$ is square. Let $A \in M_{m,n}(\R)$. Then $A$ is called \emph{totally positive} (\emph{totally nonnegative}) if all of its minors are positive (nonnegative). The subset of $M_{m,n}(\R)$ of totally positive matrices (totally nonnegative matrices) will be denoted by $TP$ ($TN$). We often write ``\emph{$A$ is $TP$}" (``\emph{$A$ is $TN$}") to refer the fact that $A$ is a totally positive matrix (totally nonnegative matrix). There are several references available about totally positive/nonnegative matrices. For example, see \cite{And87, fallat2022totally, GK02, Kar68, Sch30}. 

If $A \in M_{m,n}(\R)$, then $A$ is called \emph{$TP_k$} (\emph{$TN_k$}) where $1\leq k \leq \min\{m,n\}$, if all of its minors up to and including order $k$ are positive (nonnegative). In particular, a matrix with positive entries is $TP_1$. Observe that $A$ is $TP$ ($A$ is $TN$) is same as $A$ is $TP_n$ ($A$ is $TN_n$) for a matrix $A\in M_n(\R)$. A \emph{contiguous minor} is a minor that corresponds to consecutive rows and consecutive columns of a given matrix. The following is a classical result. Not all minors need be checked to verify $TP$.

\begin{theorem}[Fekete \cite{Fek13}] \label{fekete}
    Let $A \in M_{m,n}(\R)$. If all contiguous minors up to and including order $k$ are positive, then  $A$ is $TP_k$.  
\end{theorem}

If $A\in M_{m,n}(\R)$ and $k\leq \min\{m,n\}$, the $\binom{m}{k}$-by-$\binom{n}{k}$ matrix of all $k$-by-$k$ minors (with index sets ordered lexicographically) of $A$ is called the \emph{$k$th compound} of $A$ and is denoted by $C_k(A)$. Note that if we choose any ordering other than lexicographic, then we obtain a different matrix. However, it is similar to the compound matrix by some permutation matrix. 

By definition if $A \in M_n(\R)$ is $TP_k$, then $C_k(A)$ is $TP_1$ for all $1\leq k\leq n$. Notice that $C_1(A) = A$ and $C_n(A)$ is a scalar equal to $\det A$. Thus, if $A$ is $TP$, then so are $C_1(A)$ and $C_n(A)$. More interestingly, if $A$ is $TP$, then $C_{n-1}(A)$ is $TP$. This follows from the basic fact that if $A$ is $TP$, then $SA^{-1}S$ is $TP$, where $S=diag(1,-1, \ldots, (-1)^{n-1})$ (see, for example, \cite{GK02}). Furthermore, basic linear algebra dictates that $SA^{-1}S$ is equal to $\frac{1}{\det A}C_{n-1}(A)$. 

A natural question is what we can say about the total positivity of $C_k(A)$ for $2\leq k\leq n-2$. Our first main result gives an answer to this question.

\begin{utheorem}\label{main_thm_1} 
Let $n\geq 4$ and let $A\in M_n(\R)$. If $A$ is $TP_{k+2}$, then $C_k(A)$ is not $TP_3$ for any $2\leq k\leq n-2$. 
\end{utheorem}

An immediate corollary to Theorem \ref{main_thm_1} is as follows. 
\begin{corollary}\label{corollary1}
     Let $n\geq 4$ and let $A\in M_n(\R)$. If $C_{k}(A)$ is $TP_3$ for some $2\leq k\leq n-2$, then $A$ is not $TP_{k+2}$. 
\end{corollary}

Another important set of matrices associated with a matrix are its Dodgson condensations. The Dodgson condensation is a method developed by
C.~Dodgson (aka.~Lewis Carol) \cite{Carrol} to compute the determinant of a square matrix. The set of matrices derived from this method are called Dodgson condensation matrices. A key ingredient of this technique is the well-known Sylvester's determinantal identity.

\begin{theorem}[Sylvester's determinantal identity \cite{HJ85}]\label{thm:sylvester_identity}
Let $A \in M_{n}(\R)$, $\alpha \subseteq \{1,2,\ldots,n\}$, and suppose $|\alpha|=k$. Define the $(n-k) $-by-$ (n-k)$
matrix $B = [b_{i,j}]$, with $i,j \in \alpha^{c}$, by setting $b_{i,j}={\rm det}A[\alpha \cup \{i\},\alpha \cup \{j\}]$, for every $i,j \in \alpha^{c}$. Then, for each
$\delta, \gamma \subset \alpha^{c}$, with $|\delta|=|\gamma|=l$,
\begin{equation}
{\rm det}B[\delta,\gamma]=({\rm det}A[\alpha])^{l-1}
{\rm det}A[\alpha \cup \delta, \alpha \cup \gamma]. 
\label{sylvester-intro}
\end{equation}
\end{theorem} 
Observe that a special case of (\ref{sylvester-intro}) is that ${\rm det}B =({\rm det}A[\alpha])^{n-k-1}{\rm det}A$. Another very useful special case is the following. Let $A \in M_{n}(\R)$ be the partitioned matrix
\begin{align*}
     A = \left[ \begin{array}{ccc} a_{11} & a_{12}^{T} & a_{13}
\\ a_{21} & A_{22} & a_{23} \\ a_{31} & a_{32}^{T} & a_{33}
\end{array} \right],
\end{align*}
where $A_{22} \in M_{n-2}(\R)$, $a_{12},a_{21},a_{23},a_{32} \in \R^{n-2}$, and
$a_{11}, a_{13}, a_{31}, a_{33}$ are scalars. Define the matrices

\begin{align*}
    B= \left[ \begin{array}{cc} a_{11} & a_{12}^{T} \\ a_{21} & A_{22}  \end{array} \right],\ C= \left[ \begin{array}{cc} a_{12}^{T} & a_{13} \\ A_{22} &
a_{23}  \end{array} \right],\ D= \left[ \begin{array}{cc} a_{21} & A_{22} \\ a_{31} & a_{32}^{T} 
\end{array} \right], \text{ and } E= \left[ \begin{array}{cc} A_{22} & a_{23} \\ a_{32}^{T} &
a_{33}  \end{array} \right].
\end{align*}
If we let $\tilde{b} ={\rm det}B$, $\tilde{c} ={\rm det}C$,
$\tilde{d} ={\rm det}D$, and $\tilde{e} ={\rm det}E$, then by
(\ref{sylvester-intro}) it follows that
\[ {\rm det} \left[ \begin{array}{cc} \tilde{b} &\tilde{c} \\
\tilde{d} & \tilde{e}  \end{array} \right] = {\rm det}A_{22}{\rm
det}A. \] 
Hence, provided ${\rm det}A_{22} \neq 0$, we have 
\[ {\rm det}A= \frac{{\rm det}B{\rm det}E-{\rm det}C{\rm det}D}{{\rm
det}A_{22}}. \]

Dodgson's condensation produces a sequence of matrices $D_i(A) \in M_{n-i}(\R)$ ($1\leq i \leq n-1$) from a given matrix $A \in M_n(\R)$ to which a recursive step is implemented resulting in a scalar whose value equal to $\det A$ (see \cite{Carrol}). This algorithm can be described as follows (see also \cite{dodgson}): Given $A=[a_{i,j}] \in M_n(\R)$, 

\begin{description}
\item[Step 1] Construct a matrix $D_1(A)=[M^{(1)}_{i,j}] \in M_{n-1}(\R)$, where for 
$i,j = 1,2, \ldots, n-1$,
\[ M^{(1)}_{i,j} = \det \left[ \begin{array}{cc} 
a_{i,j} & a_{i, j+1}  \\ a_{i+1,j} & a_{i+1,j+1} \end{array} \right];\]

\item[Step 2] Using both $A$ and $D_1(A)$, construct another matrix $D_2(A)=[M^{(2)}_{i,j}] \in M_{n-2}(A)$, where for
$i,j = 1,2, \ldots, n-2$,
\[ M^{(2)}_{i,j} = \det \left[ \begin{array}{cc} 
M^{(1)}_{i,j} &M^{(1)}_{i,j+1}  \\ M^{(1)}_{i+1,j} & M^{(1)}_{i+1,j+1} \end{array} \right]/a_{i+1,j+1}; \]

\item[Step 3] Replace $A$ with $D_1(A)$, $D_1(A)$ with $D_2(A)$, and repeat Step 2, until a scalar is produced.
\end{description}
The above algorithm is a simple mechanism for computing the determinant of a matrix. We refer to a matrix $D_k(A)$ as a \emph{Dodgson condensation matrix}. Applying Sylvester's identity (Theorem \ref{thm:sylvester_identity}) and following the discussion afterwards, it can be seen that the entries of $D_2(A)$ are in fact equal to 
\[ M_{i,j}^{(2)} = \det A[\{i,i+1,i+2\},\{j,j+1,j+2\}].\]
In other words, the entries of $D_2(A)$ are precisely all of the contiguous minors of order $3$ of $A$. In fact, by using the Sylvester's identity repeatedly, the above claim holds more generally as follows.
\begin{lemma} \label{Dk-entries}
    Let $A \in M_n(\R)$ and let $D_k(A) = [M_{i,j}^{(k)}]$. Then for any $i,j=1,2,\ldots, n-k$ we have 
    \[ M_{i,j}^{(k)} = \det A[\{i,\ldots, i+k\},\{j,\ldots, j+k\}].\] 
\end{lemma}
Hence, the matrix $D_k(A)$ consists of all contiguous minors of order $k+1$ of $A$. So if $A$ is $TP_{k+1}$, then $D_{k}(A)$ is $TP_1$ for any $1\leq k\leq n-1$. In fact, we can say more about the total positivity of a Dodgson condensation matrix. That is our second main result.  

\begin{utheorem}\label{main_thm_2}
    Let $A \in M_n(\R)$ and let $k \geq 1$ be an integer. Then the following statements hold:
    \begin{enumerate}
         \item If $A$ is $TP_{k+2}$, then $D_k(A)$ is $TP_2$.
          \item If $A$ is $TP_{k+3}$, then $D_k(A)$ is $TP_3$.
    \end{enumerate}
\end{utheorem}

We demonstrate that the previous result can not be improved for others values of $k$ (see Example \ref{eg_for_Thm_B}). However, we prove a much broader result for a class of matrices known as Hankel matrices. For a given sequence of scalars $a_0,a_1,\ldots,a_{2n}$, the $(n+1) $-by-$ (n+1)$ symmetric matrix 
\begin{align*}
    A = (a_{i+j})_{i,j=0}^n = \begin{bmatrix}
        a_0 & a_1 & \ldots & a_n\\
        a_1 & a_2 & \ldots & a_{n+1}\\
        \vdots & \vdots & \ddots & \vdots\\
        a_n & a_{n+1} & \ldots & a_{2n}
    \end{bmatrix}
\end{align*}
is called a \emph{Hankel matrix}. Each Dodgson condensation matrix corresponding to a totally positive Hankel matrix is totally positive. 
\begin{utheorem}\label{main_thm_Hankel}
    Let $A$ be a Hankel matrix based on the sequence of scalars $a_0,a_1,\ldots,a_{2n}$. If $A$ is $TP$, then $D_k(A)$ is a Hankel and $TP$ matrix for all $1 \leq k \leq n$.  
\end{utheorem}

An important goal in the study of totally positive matrices is to find out sufficient conditions to verify the  total positivity of a matrix. A matrix $A = [a_{i,j}] \in M_{n}(\R)$ is called $TP_2(c)$, if there exists a positive number $c$ such that $a_{i,j}a_{i+1,j+1} \geq c a_{i+1,j}a_{i,j+1}$ for all $i,j = 1,\ldots,n$. In 2006, Katkova and Vishnyakova \cite{katkova2006sufficient} proved the following fascinating result, which may be viewed as an improvement of the work of Craven and Csordas \cite{CC98}.
  
\begin{theorem}[Katkova and Vishnyakova \cite{katkova2006sufficient}]
    Let $A \in M_{n}(\R)$ and $c \geq 4\cos^2\left(\frac{\pi}{n+1}\right)$ be a positive number. If $A$ is $TP_2(c)$, then $A$ is $TP$. \label{OA}
\end{theorem}

We desire to determine if we can learn anything about whether a matrix $A$ is totally positive by looking at the total positivity of compounds or Dodgson condensation matrices. Suppose a matrix $A$ is $TP_k$. If $C_{k+1}(A)$ is $TP_1$, then by definition $A$ is $TP_{k+1}$. On the other hand, by Fekete's Theorem \ref{fekete}, if $D_{k}(A)$ is $TP_1$, then $A$ is $TP_{k+1}$. In fact, even more can be said, if $A$ is $TP_1$ and $C_2(A)$ is TP, then it follows that $A$ must also be TP. However, this fact, which may be seen to serve as a possible generalization of the sufficiency-type conditions from Theorem \ref{OA}, is vacuous as can be seen from Corollary \ref{corollary1}. Note that by Corollary \ref{corollary1} if a certain compound matrix is $TP_3$ or more, then $A$ can not be totally positive. A natural question is what we can say if we know that a compound matrix or a Dodgson condensation matrix is $TP_2$. Our final main result asserts an answer to this question.  

\begin{utheorem}\label{main_thm_3}
    Let $A \in M_{n}(\R)$ be $TP_k$ for some $1\leq k \leq n-2$. If $D_{k}(A)$ is $TP_2$, then $A$ is $TP_{k+2}$.
\end{utheorem}

Since $D_k(A)$ is a submatrix of $C_{k+1}(A)$ we obtain an immediate consequence of Theorem \ref{main_thm_3}. 

\begin{corollary}\label{corollary2}
    Let $A \in M_{n}(\R)$ be $TP_k$ for some $1\leq k \leq n-2$. If $C_{k+1}(A)$ is $TP_2$, then $A$ is $TP_{k+2}$.
\end{corollary}

The rest of the paper is organized as follows. In Section \ref{sec:prelim}, we provide some preliminary tools needed for the proof of Theorem \ref{main_thm_1}. In Section \ref{sec:proofs_and_remarks}, we 
provide proofs of our four main theorems, together with some remarks and examples.

\section{Preliminaries}\label{sec:prelim}
We now consider an important and very useful correspondence between totally nonnegative matrices and a corresponding bidiagonal factorization. Recall that a {\em bidiagonal factorization} of a matrix $A$, is simply a decomposition of $A$ into a product of bidiagonal matrices. A matrix $B=[b_{i,j}]$ is called {\em lower (upper) bidiagonal} if $b_{i,j}=0$, whenever $i<j$ and $i > j-1$ ($i>j$ and $i < j-1$). We assume throughout that all matrices are square of order $n$. Let $I$ denote the identity matrix. For $1\leq i,j \leq n$, we let $E_{ij}$  being the elementary standard basis matrix whose only nonzero entry is a $1$ in the $(i,j)$ position. Let us define $L_i(\ell) := I + \ell E_{i,i-1}$ and $U_j(u) := I+uE_{j-1,j}$ where $2\leq i,j \leq n$ and $\ell,u\in \R$. If a matrix is of the form $L_i(\ell)$ or $U_j(u)$, then it is called an \emph{elementary bidiagonal matrix}. We may shorten $L_{i}(\ell)$ to $L_{i}$ or $U_j(u)$ to $U_j$ when the choice of numbers $\ell,u \in \R$ are not so important. 

The bidiagonal factorization associated with totally nonnegative matrices has had a long and productive history, for example, see \cite{Cry73, Cry76, Fal01, GP94, Loe55, Whi52}. 

\begin{theorem}\label{thm:bidiag} \cite{GP92, Fal01}
Let $A$ be an $n $-by-$ n$ nonsingular totally nonnegative matrix.
Then $A$ can be written as
\begin{align}\label{bidiag}
 A &= (L_2(\ell_1)) \cdot (L_{3}(\ell_{2}) L_2(\ell_{3})) \cdots
 \cdot(L_n(\ell_{k-n+1}) \cdots L_2(\ell_{k}))\cdot \nonumber\\ 
 &\ \ \ D\cdot (U_2(u_{k})\cdots U_{n-1}(u_{k-n+2})U_n(u_{k-n+1}))\cdots (U_{2}(u_3)U_3(u_2)) \cdot U_2(u_1),
\end{align}
where $k=\binom{n}{2}$; $\ell_{i},u_{j} \geq 0$ for all $i,j \in
\{1,2,\ldots, k\}$; and $D = \text{diag}(d_1,d_2,\ldots,d_n)$ is a diagonal matrix with $d_i > 0$.
\label{bifact}
\end{theorem}

\begin{figure}[h]
\centering
\begin{tikzpicture}
\filldraw [black] (0,0) circle (3pt) node [left] {$\ 1\ $};
\filldraw [black] (12,0) circle (3pt) node [right] {$\ 1\ $};
\filldraw [black] (0,1) circle (3pt) node [left] {$\ 2\ $};
\filldraw [black] (12,1) circle (3pt) node [right] {$\ 2\ $};
\filldraw [black] (0,2) circle (3pt) node [left] {$\ 3\ $};
\filldraw [black] (12,2) circle (3pt) node [right] {$\ 3\ $};
\filldraw [black] (0,5) circle (3pt) node [left] {$\ n-2\ $};
\filldraw [black] (12,5) circle (3pt)node [right] {$\ n-2\ $};
\filldraw [black] (0,6) circle (3pt) node [left] {$\ n-1\ $};
\filldraw [black] (12,6) circle (3pt) node [right] {$\ n-1\ $};
\filldraw [black] (0,7) circle (3pt) node [left] {$\ n\ $};
\filldraw [black] (12,7) circle (3pt) node [right] {$\ n\ $};
\filldraw [black] (6,0.1) node [above] {$d_1$};
\filldraw [black] (6,1.1) node [above] {$d_2$};
\filldraw [black] (6,2.1) node [above] {$d_3$};
\filldraw [black] (6,5.1) node [above] {$d_{n-2}$};
\filldraw [black] (6,6.1) node [above] {$d_{n-1}$};
\filldraw [black] (6,7.1) node [above] {$d_n$};
\filldraw [black] (1.2,0.5) node [left] {$\ell_1$};
\filldraw [black] (10.8,0.5) node [right] {$u_1$};
\filldraw [black] (1.65,1.5) node [left] {$\ell_2$};
\filldraw [black] (10.3,1.5) node [right] {$u_2$};
\filldraw [black] (2.15,0.5) node [left] {$\ell_3$};
\filldraw [black] (9.8,0.5) node [right] {$u_3$};
\filldraw [black] (1.65,5.5) node [left] {$\ell_{k-2n+2}$};
\filldraw [black] (10.3,5.5) node [right] {$u_{k-2n+2}$};
\filldraw [black] (3.7,1.5) node [left] {$\ell_{k-n-1}$};
\filldraw [black] (8.22,1.5) node [right] {$u_{k-n-1}$};
\filldraw [black] (4.2,0.5) node [left] {$\ell_{k-n}$};
\filldraw [black] (7.82,0.5) node [right] {$u_{k-n}$};
\filldraw [black] (2.3,6.5) node [right] {$\ell_{k-n+1}$};
\filldraw [black] (9.65,6.5) node [left] {$u_{k-n+1}$};
\filldraw [black] (2.8,5.5) node [right] {$\ell_{k-n+2}$};
\filldraw [black] (9.15,5.5) node [left] {$u_{k-n+2}$};
\filldraw [black] (4.75,1.5) node [left] {$\ell_{k-1}$};
\filldraw [black] (7.2,1.5) node [right] {$u_{k-1}$};
\filldraw [black] (5.15,0.5) node [left] {$\ell_{k}$};
\filldraw [black] (6.8,0.5) node [right] {$u_{k}$};
\draw [line width=0.5mm,dotted] (0,2.7) -- (0,4.5);
\draw [line width=0.5mm,dotted] (12,2.7) -- (12,4.5);
\draw [line width=0.5mm,dotted] (6,2.7) -- (6,4.5);
\draw [line width=0.5mm,dotted] (2.5,0.5) -- (2.8,0.5);
\draw [line width=0.5mm,dotted] (2,1.5) -- (2.3,1.5);
\draw [line width=0.5mm,dotted] (9.2,0.5) -- (9.5,0.5);
\draw [line width=0.5mm,dotted] (9.7,1.5) -- (10,1.5);
\draw [line width=0.5mm,black]  (0,0) -- (12,0);
\draw [line width=0.5mm,black]  (0,1) -- (12,1);
\draw [line width=0.5mm,black]  (0,2) -- (12,2);
\draw [line width=0.5mm,black]  (0,5) -- (12,5);
\draw [line width=0.5mm,black]  (0,6) -- (12,6);
\draw [line width=0.5mm,black]  (0,7) -- (12,7);
\draw [line width=0.5mm,black]  (2,7) -- (3,5);
\draw [line width=0.5mm,black,dotted]  (3.2,4.5) -- (4.1,2.7);
\draw [line width=0.5mm,black]  (4.5,2) -- (5.5,0);
\draw [line width=0.5mm,black]  (1.5,6) -- (2,5);
\draw [line width=0.5mm,black,dotted]  (2.2,4.5) -- (3.1,2.7);
\draw [line width=0.5mm,black]  (3.5,2) -- (4.5,0);
\draw [line width=0.5mm,black]  (1.5,2) -- (2.5,0);
\draw [line width=0.5mm,black]  (1,1) -- (1.5,0);
\draw [line width=0.5mm,black]  (10,7) -- (9,5);
\draw [line width=0.5mm,black,dotted]  (8.8,4.5) -- (7.9,2.7);
\draw [line width=0.5mm,black]  (7.5,2) -- (6.5,0);
\draw [line width=0.5mm,black]  (10.5,6) -- (10,5);
\draw [line width=0.5mm,black,dotted]  (9.8,4.5) --(8.9,2.7);
\draw [line width=0.5mm,black]  (8.5,2) -- (7.5,0);
\draw [line width=0.5mm,black]  (10.5,2) -- (9.5,0);
\draw [line width=0.5mm,black]  (11,1) -- (10.5,0);
\draw [-stealth,line width=1mm] (0.6,0) -- (0.8,0);
\draw [-stealth,line width=1mm] (0.6,1) -- (0.8,1);
\draw [-stealth,line width=1mm] (0.6,2) -- (0.8,2);
\draw [-stealth,line width=1mm] (0.6,5) -- (0.8,5);
\draw [-stealth,line width=1mm] (0.6,6) -- (0.8,6);
\draw [-stealth,line width=1mm] (0.6,7) -- (0.8,7);
\draw [-stealth,line width=1mm] (5.95,0) -- (6.15,0);
\draw [-stealth,line width=1mm] (5.95,1) -- (6.15,1);
\draw [-stealth,line width=1mm] (5.95,2) -- (6.15,2);
\draw [-stealth,line width=1mm] (5.95,5) -- (6.15,5);
\draw [-stealth,line width=1mm] (5.95,6) -- (6.15,6);
\draw [-stealth,line width=1mm] (5.95,7) -- (6.15,7);
\draw [-stealth,line width=1mm] (11.4,0) -- (11.6,0);
\draw [-stealth,line width=1mm] (11.4,1) -- (11.6,1);
\draw [-stealth,line width=1mm] (11.4,2) -- (11.6,2);
\draw [-stealth,line width=1mm] (11.4,5) -- (11.6,5);
\draw [-stealth,line width=1mm] (11.4,6) -- (11.6,6);
\draw [-stealth,line width=1mm] (11.4,7) -- (11.6,7);
\draw [-stealth,line width=1mm] (1.25,0.5) -- (1.36,0.3);
\draw [-stealth,line width=1mm] (2.25,0.5) -- (2.36,0.3);
\draw [-stealth,line width=1mm] (4.25,0.5) -- (4.36,0.3);
\draw [-stealth,line width=1mm] (5.25,0.5) -- (5.36,0.3);
\draw [-stealth,line width=1mm] (1.735,1.5) -- (1.855,1.3);
\draw [-stealth,line width=1mm] (3.735,1.5) -- (3.855,1.3);
\draw [-stealth,line width=1mm] (4.735,1.5) -- (4.855,1.3);
\draw [-stealth,line width=1mm] (1.735,5.5) -- (1.855,5.3);
\draw [-stealth,line width=1mm] (2.735,5.5) -- (2.855,5.3);
\draw [-stealth,line width=1mm] (2.25,6.5) -- (2.36,6.3);
\draw [-stealth,line width=1mm] (6.7,0.4) -- (6.77,0.53);
\draw [-stealth,line width=1mm] (7.7,0.4) -- (7.77,0.53);
\draw [-stealth,line width=1mm] (9.7,0.4) -- (9.77,0.53);
\draw [-stealth,line width=1mm] (10.7,0.4) -- (10.77,0.53);
\draw [-stealth,line width=1mm] (7.2,1.4) -- (7.27,1.53);
\draw [-stealth,line width=1mm] (8.2,1.4) -- (8.27,1.53);
\draw [-stealth,line width=1mm] (10.2,1.4) -- (10.27,1.53);
\draw [-stealth,line width=1mm] (9.2,5.4) -- (9.27,5.53);
\draw [-stealth,line width=1mm] (10.2,5.4) -- (10.27,5.53);
\draw [-stealth,line width=1mm] (9.7,6.4) -- (9.77,6.53);
\end{tikzpicture}
\caption{Planar network corresponding to the bidiagonal factorization  in Theorem \ref{thm:bidiag}}
\label{fig:enter-label}
\end{figure}
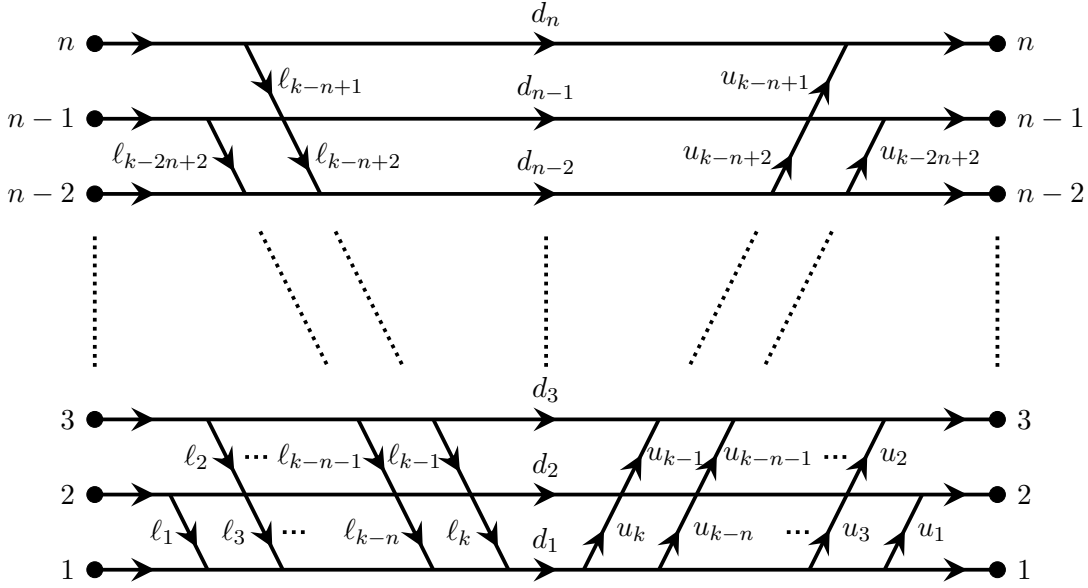

One of the tools used in \cite{Fal01, FZ00} is a representation of a bidiagonal factorization (see also \cite{Fal01,GP92}) in terms of  planar networks, which is well known (see e.g. \cite{Bre95, Fal01}). Recall that a {\em planar network} is simply a weighted directed graph without cycles. The planar networks of interest and corresponding to the factorization in (\ref{bidiag}) is represented as in Figure \ref{fig:enter-label}, and contain $n$ sources (labelled vertices on the left) and $n$ sinks (labelled vertices on the right). Note that each horizontal edge in this planar network has a weight of 1, except for the horizontal edges corresponding to the diagonal factor $D$. 

Now, given a planar network $\Gamma$ such as the one in Figure \ref{fig:enter-label} we have  following important and tremendously useful
fact, which is sometimes also referred to as Lindstrom's Lemma (see \cite{FZ00}).
For index sets $\alpha=\{i_{1}, i_{2}, \ldots, i_{t}\}$ and
$\beta=\{j_{1},j_{2},\ldots,j_{t}\}$, consider a collection
$P(\alpha,\beta)$ of all families of vertex-disjoint paths joining
the vertices $\{i_{1}, i_{2}, \ldots, i_{t}\}$ on the left of the
network $\Gamma$ with the vertices $\{j_{1},j_{2},\ldots,j_{t}\}$ on
the right. In particular, it follows that $i_r$ must be joined to $j_r$ by a path for each $r=1,2,\ldots, t$. For $\pi\in P(\alpha,\beta)$, let $w(\pi)$ be the
product of all the weights assigned to edges that form a family
$\pi$. Then
$$\displaystyle {
\det A[\alpha,\beta]=\sum_{\pi\in P(\alpha,\beta)} w(\pi)\ . }
$$

Theorem \ref{bifact} implies that every invertible TN matrix can be represented by a weighted planar network. More general weighted planar networks can also be associated to TN matrices (see e.g. \cite{Bre95, Fal01}). Often we may interchange the notions of bidiagonal factorizations and planar networks when dealing with various properties of 
TN matrices. In particular, we will often think of a TN matrix by merely representing it as a general planar network. Given such a correspondence, we can recast Theorem \ref{bifact} and other related results
in terms of planar networks. 

\begin{theorem}
If $\Gamma$ is a planar network corresponding to (\ref{bidiag}), then:
\begin{enumerate}
\item The associated matrix is nonsingular and totally
nonnegative if and only if all the parameters are nonnegative and all $d_{i}>0$;
 
\item The associated matrix is totally positive if and only if all the 
parameters are positive and all $d_{i}>0$;
\end{enumerate}\label{networkcase}
\end{theorem}

\section{Proof of the main results and related comments}\label{sec:proofs_and_remarks}
We start this section by proving Theorem \ref{main_thm_1}. 
\begin{proof}[Proof of Theorem \ref{main_thm_1}] We first prove the statement for $k = n-2$. So assume $A$ is $TP$, and we want to prove that $C_{n-2}(A)$ is not $TP_3$. Consider the following sets each of size $n-2$:
\begin{align*}
S_1 &= \{1,2,\ldots,n-4,n-3,n-2\},\\
S_2 &= \{1,2,\ldots,n-4,n-3,n-1\},\\
S_3 &= \{1,2,\ldots,n-4,n-3,n\},\\
S_4 &= \{1,2,\ldots,n-4,n-2,n-1\}.    
\end{align*}
Let us define an $4$-by-$ 4$ matrix $S:= (s_{i,j})$ where $s_{i,j} := \det(A[S_i,S_j])$ for $1\leq i,j\leq 4$. Notice that $S$ is the leading principal submatrix of $C_{n-2}(A)$. Now, assume that $A$ is represented as a planar network according to Theorem \ref{networkcase}. Given the nature of the sets $S_i$ ($i=1,2,3,4$), so when computing the entries of $S$, it is sufficient to just label the top-part of the planar network representing $A$ (see Figure \ref{PN-a}). For the purpose of brevity in the computation below, we assume that the diagonal factor in (\ref{bidiag}) is the identity. The strict negativity or positivity of the minors in the matrix $S$ identified below are not affected by introducing an arbitrary positive diagonal factor.

\begin{figure}[h]
\centering
\begin{tikzpicture}
\filldraw [black] (-0.5,1) circle (3pt) node [left] {$\ 1\ $};
\filldraw [black] (12.5,1) circle (3pt) node [right] {$\ 1\ $};
\filldraw [black] (-0.5,2) circle (3pt) node [left] {$\ 2\ $};
\filldraw [black] (12.5,2) circle (3pt) node [right] {$\ 2\ $};
\filldraw [black] (-0.5,5) circle (3pt) node [left] {$\ n-3\ $};
\filldraw [black] (12.5,5) circle (3pt)node [right] {$\ n-3\ $};
\filldraw [black] (-0.5,6) circle (3pt) node [left] {$\ n-2\ $};
\filldraw [black] (12.5,6) circle (3pt) node [right] {$\ n-2\ $};
\filldraw [black] (-0.5,7) circle (3pt) node [left] {$\ n-1\ $};
\filldraw [black] (12.5,7) circle (3pt) node [right] {$\ n-1\ $};
\filldraw [black] (-0.5,8) circle (3pt) node [left] {$\ n\ $};
\filldraw [black] (12.5,8) circle (3pt) node [right] {$\ n\ $};
\filldraw [black] (0.8,5.5) node [right] {$a$};
\filldraw [black] (1.8,5.5) node [right] {$c$};
\filldraw [black] (11,5.5) node [left] {$\ell$};
\filldraw [black] (10,5.5) node [left] {$j$};
\filldraw [black] (1.8,7.5) node [right] {$d$};
\filldraw [black] (1.3,6.5) node [right] {$b$};
\filldraw [black] (2.3,6.5) node [right] {$e$};
\filldraw [black] (9.5,6.5) node [left] {$h$};
\filldraw [black] (10.5,6.5) node [left] {$k$};
\filldraw [black] (10,7.5) node [left] {$i$};
\filldraw [black] (2.8,5.5) node [right] {$f$};
\filldraw [black] (9.1,5.5) node [left] {$g$};
\draw [line width=0.5mm,dotted] (-0.5,2.7) -- (-0.5,4.5);
\draw [line width=0.5mm,dotted] (12.5,2.7) -- (12.5,4.5);
\draw [line width=0.5mm,dotted] (6,2.7) -- (6,4.5);
\draw [line width=0.5mm,dotted] (1.45,1.5) -- (2,1.5);
\draw [line width=0.5mm,dotted] (10.05,1.5) -- (10.6,1.5);
\draw [line width=0.5mm,black]  (-0.5,1) -- (12.5,1);
\draw [line width=0.5mm,black]  (-0.5,2) -- (12.5,2);
\draw [line width=0.5mm,black]  (-0.5,5) -- (12.5,5);
\draw [line width=0.5mm,black]  (-0.5,6) -- (12.5,6);
\draw [line width=0.5mm,black]  (-0.5,7) -- (12.5,7);
\draw [line width=0.5mm,black]  (-0.5,8) -- (12.5,8);
\draw [line width=0.5mm,black]  (1.5,8) -- (2,7);
\draw [line width=0.5mm,black]  (2,7) -- (3,5);
\draw [line width=0.5mm,black,dotted]  (3.2,4.5) -- (4.1,2.7);
\draw [line width=0.5mm,black]  (4.5,2) -- (5,1);
\draw [line width=0.5mm,black]  (1,7) -- (1.5,6);
\draw [line width=0.5mm,black]  (1.5,6) -- (2,5);
\draw [line width=0.5mm,black,dotted]  (2.2,4.5) -- (3.1,2.7);
\draw [line width=0.5mm,black]  (3.5,2) -- (4,1);
\draw [line width=0.5mm,black]  (0.5,6) -- (1,5);
\draw [line width=0.5mm,black,dotted]  (1.2,4.5) -- (2.1,2.7);
\draw [line width=0.5mm,black]  (2.5,2) -- (3,1);
\draw [line width=0.5mm,black]  (0.5,2) -- (1,1);
\draw [line width=0.5mm,black]  (10,7) -- (10.5,8);
\draw [line width=0.5mm,black]  (10,7) -- (9,5);
\draw [line width=0.5mm,black,dotted]  (8.8,4.5) -- (7.9,2.7);
\draw [line width=0.5mm,black]  (7.5,2) -- (7,1);
\draw [line width=0.5mm,black]  (10.5,6) -- (11,7);
\draw [line width=0.5mm,black]  (10.5,6) -- (10,5);
\draw [line width=0.5mm,black,dotted]  (9.8,4.5) --(8.9,2.7);
\draw [line width=0.5mm,black]  (8.5,2) -- (8,1);
\draw [line width=0.5mm,black]  (11,5) -- (11.5,6);
\draw [line width=0.5mm,black,dotted]  (10.8,4.5) --(9.9,2.7);
\draw [line width=0.5mm,black]  (9.5,2) -- (9,1);
\draw [line width=0.5mm,black]  (11.5,2) -- (11,1);
\draw [-stealth,line width=1mm] (0.1,1) -- (0.3,1);
\draw [-stealth,line width=1mm] (0.1,2) -- (0.3,2);
\draw [-stealth,line width=1mm] (0.1,5) -- (0.3,5);
\draw [-stealth,line width=1mm] (0.1,6) -- (0.3,6);
\draw [-stealth,line width=1mm] (0.1,7) -- (0.3,7);
\draw [-stealth,line width=1mm] (0.1,8) -- (0.3,8);
\draw [-stealth,line width=1mm] (5.95,1) -- (6.15,1);
\draw [-stealth,line width=1mm] (5.95,2) -- (6.15,2);
\draw [-stealth,line width=1mm] (5.95,5) -- (6.15,5);
\draw [-stealth,line width=1mm] (5.95,6) -- (6.15,6);
\draw [-stealth,line width=1mm] (5.95,7) -- (6.15,7);
\draw [-stealth,line width=1mm] (5.95,8) -- (6.15,8);
\draw [-stealth,line width=1mm] (11.9,1) -- (12.1,1);
\draw [-stealth,line width=1mm] (11.9,2) -- (12.1,2);
\draw [-stealth,line width=1mm] (11.9,5) -- (12.1,5);
\draw [-stealth,line width=1mm] (11.9,6) -- (12.1,6);
\draw [-stealth,line width=1mm] (11.9,7) -- (12.1,7);
\draw [-stealth,line width=1mm] (11.9,8) -- (12.1,8);
\draw [-stealth,line width=1mm] (0.735,1.5) -- (0.855,1.3);
\draw [-stealth,line width=1mm] (2.735,1.5) -- (2.855,1.3);
\draw [-stealth,line width=1mm] (3.735,1.5) -- (3.855,1.3);
\draw [-stealth,line width=1mm] (4.735,1.5) -- (4.855,1.3);
\draw [-stealth,line width=1mm] (0.735,5.5) -- (0.855,5.3);
\draw [-stealth,line width=1mm] (1.735,5.5) -- (1.855,5.3);
\draw [-stealth,line width=1mm] (2.735,5.5) -- (2.855,5.3);
\draw [-stealth,line width=1mm] (1.25,6.5) -- (1.36,6.3);
\draw [-stealth,line width=1mm] (2.25,6.5) -- (2.36,6.3);
\draw [-stealth,line width=1mm] (1.75,7.5) -- (1.86,7.3);
\draw [-stealth,line width=1mm] (7.2,1.4) -- (7.27,1.53);
\draw [-stealth,line width=1mm] (8.2,1.4) -- (8.27,1.53);
\draw [-stealth,line width=1mm] (9.2,1.4) -- (9.27,1.53);
\draw [-stealth,line width=1mm] (11.2,1.4) -- (11.27,1.53);
\draw [-stealth,line width=1mm] (9.2,5.4) -- (9.27,5.53);
\draw [-stealth,line width=1mm] (10.2,5.4) -- (10.27,5.53);
\draw [-stealth,line width=1mm] (11.2,5.4) -- (11.27,5.53);
\draw [-stealth,line width=1mm] (9.7,6.4) -- (9.77,6.53);
\draw [-stealth,line width=1mm] (10.7,6.4) -- (10.77,6.53);
\draw [-stealth,line width=1mm] (10.2,7.4) -- (10.27,7.53);
\end{tikzpicture}
\caption{Planar network representing $A$.}
 \label{PN-a}
\end{figure}
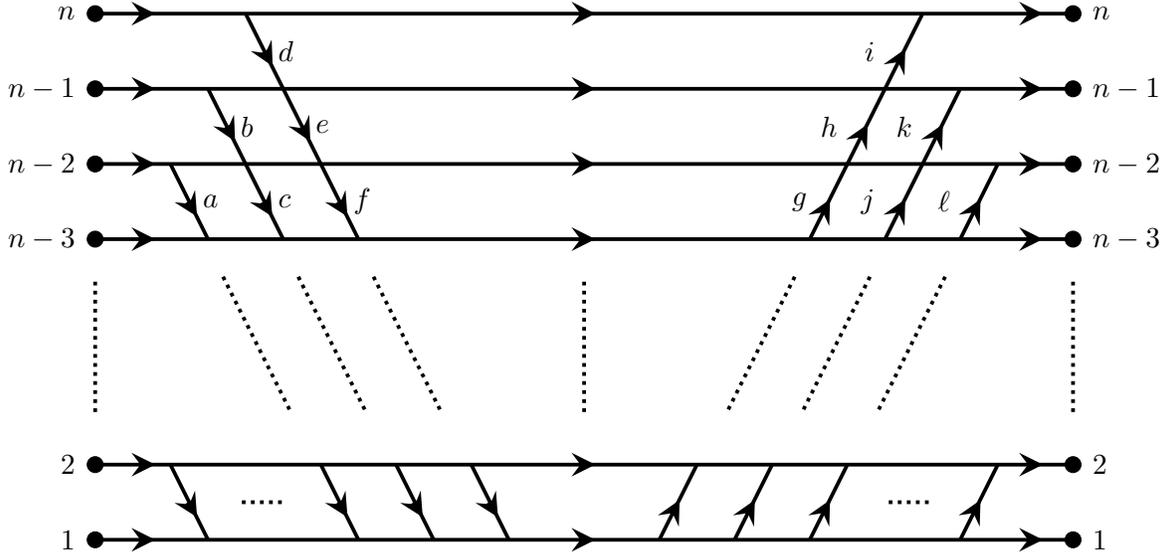
By examining the disjoint paths in the planar network, Figure \ref{PN-a}, the entries of $S$ are as follows:
\begin{align*}
s_{1,1} &= 1,\ \ \  s_{1,2} = h+k,\ \ \ s_{1,3} = hi,\ \ \  s_{1,4} = jh + \ell h + \ell k,\\
s_{2,1} &= b + e,\ \ \  s_{2,2} = b h + e h + b k + e k + 1,\ \ \ s_{2,3} = b h i + e h i + i
,\\
s_{2,4} &= b h j + e h j + b h \ell + e h \ell + b k \ell + e k \ell + g + j + \ell,\\
s_{3,1} &= d e,\ \ \ s_{3,2} = d e h + d e k + d,\ \ \  s_{3,3} = d e h i + d i + 1,\\
s_{3,4} &= d e h j + d e h \ell + d e k \ell + d g + d j + d \ell,\\
s_{4,1} &= a b + a e + c e,\ \ \ s_{4,2} = a b h + a e h + c e h + a b k + a e k + c e k + a + c + f,\\
s_{4,3} &= a b h i + a e h i + c e h i + a i + c i + f i,\\
s_{4,4} &= a b h j + a e h j + c e h j +a b h \ell + a e h \ell + c e h \ell + a b k \ell + a e k \ell\ +\\ 
&\hspace*{0.5cm} c e k \ell + a g + c g + f g + a j + c j + f j + a \ell + c \ell + f \ell + 1.
\end{align*}
After a straightforward computation we obtain that  
\begin{align}\label{3_by_3_minors}
    \det(S[\{1,2,3\},\{1,3,4\}]) &= -g-j-\ell < 0, \nonumber\\
    \det(S[\{1,2,3\},\{2,3,4\}]) &= -gh-gk-jk<0,\nonumber\\
    \det(S[\{1,2,4\},\{1,3,4\}]) &= i  >0,\nonumber\\
    \det(S[\{1,2,4\},\{2,3,4\}])&= ik >0,
\end{align}
and 
\begin{align}\label{another3_by_3_minors}
    \det(S[\{1,3,4\},\{1,2,3\}]) &= -a-c-f < 0, \nonumber\\
    \det(S[\{2,3,4\},\{1,2,3\}]) &= -bc-bf-ef<0,\nonumber\\
    \det(S[\{1,3,4\},\{1,2,4\}]) &= d  >0,\nonumber\\
    \det(S[\{2,3,4\},\{1,2,4\}])&= bd >0.
\end{align}
This implies that $S$ is not $TP_3$, and hence, $C_{n-2}(A)$ is not $TP_3$. Now, we prove the statement for $2\leq k\leq n-3$. So assume $A$ is $TP_{k+2}$ and consider any submatrix $B$ of $A$ of order $k+2$. Notice that, $B$ is $TP$, and $C_k(B)$ is a submatrix of $C_k(A)$. By the previous case $C_k(B)$ is not $TP_3$, and hence $C_k(A)$ is not $TP_3$. This completes the proof. 
\end{proof}

\begin{remark}\label{rem:any_ordering_of_compounds}
    Recall that in the definition of compound matrix $C_k(A)$ we use lexicographic ordering of subsets. One may wonder whether using a different ordering of subsets could alter the result of Theorem \ref{main_thm_1}. Unfortunately, that is not the case, since in any ordering at least one of the minors from (\ref{3_by_3_minors}) is negative. 
    
    Indeed, any permutation of rows and columns, simultaneously, of the matrix $C_{n-2}(A)$ induces a permutation on the rows and columns of the matrix $S$. Considering all $4!$ permutations and using the fact that an odd permutation changes the sign of a determinant while an even permutation does not, we can conclude that in any permutation at least one of the minors from (\ref{3_by_3_minors}) is negative. 
\end{remark}

We provide an example which shows that the result of Theorem \ref{main_thm_1} can not be improved to ``not $TP_2$''. 

\begin{example}\label{example_compound_can_be_TP2}
    Let us consider the following matrix
    \begin{align*}
        A = \begin{bmatrix}
            1/2 & 1/3 & 1/4 & 1/5\\
            1/3 & 1/4 & 1/5 & 1/6\\
            1/4 & 1/5 & 1/6 & 1/7\\
            1/5 & 1/6 & 1/7 & 1/8
        \end{bmatrix}.
    \end{align*}
    Its second compound is given by 
    \begin{align*}
       \displaystyle{ C_2(A) = \begin{bmatrix}
{1}/{72} & {1}/{60} & {1}/{60} & {1}/{240} & {1}/{180} & {1}/{600} \\
{1}/{60} & {1}/{48} & {3}/{140} & {1}/{180} & {4}/{525} & {1}/{420} \\
{1}/{60} & {3}/{140} & {9}/{400} & {1}/{168} & {1}/{120} & {3}/{1120} \\
{1}/{240} & {1}/{180} & {1}/{168} & {1}/{600} & {1}/{420} & {1}/{1260} \\
{1}/{180} & {4}/{525} & {1}/{120} & {1}/{420} & {1}/{288} & {1}/{840} \\
{1}/{600} & {1}/{420} & {3}/{1120} & {1}/{1260} & {1}/{840} & {1}/{2352}
        \end{bmatrix}}
    \end{align*}
One can check that $A$ is $TP_4$, and, moreover $C_2(A)$ is also $TP_2$. 
\end{example}

We note in passing (see \cite{MR}) that the example of a $4$-by-$4$ totally nonnegative Green's matrix provided in \cite{MR} does not have a totally nonnegative second compound matrix as claimed. In fact, the $2$-by-$2$ minor in rows $\{2,3\}$ and columns $\{3,4\}$ is negative.

However, computationally it seems rare to find totally positive matrices such that all of its compounds are $TP_2$. It would be interesting to find out a characterization or an infinite family of totally positive matrices such that all of its compound matrices are $TP_2$. 

Next we prove our second main theorem. 

\begin{proof}[Proof of Theorem \ref{main_thm_2}]
  For the first statement, assume $A$ is $TP_{k+2}$. We know that $D_k(A)$ is $TP_1$. Now, consider an arbitrary $2 $-by-$ 2$ contiguous submatrix of $D_k(A)$ indexed by the rows $\{i,i+1\}$ and columns $\{j,j+1\}$. Then we have 
    \begin{equation}
        \det D_k(A)[\{i,i+1\},\{j,j+1\}] = M_{i,j}^{(k)}M_{i+1,j+1}^{(k)}-M_{i,j+1}^{(k)}M_{i+1,j}^{(k)} =  M_{i,j}^{(k+1)}M_{i+1,j+1}^{(k-1)}. \label{2x2eqn}\end{equation}
Hence, it follows by Lemma \ref{Dk-entries} that
    \begin{eqnarray*}
        \det D_k(A)[\{i,i+1\},\{j,j+1\}] & = & 
    \det A[\{i,\ldots, i+k+1\},\{j,\ldots, j+k+1\}] \\ & & \hspace*{2cm} \cdot \det A[\{i+1,\ldots, i+k\},\{j+1,\ldots, j+k\}].
    \end{eqnarray*}
   Both   $\det A[\{i,\ldots, i+k+1\},\{j,\ldots, j+k+1\}]$ and $\det A[\{i+1,\ldots, i+k\},\{j+1,\ldots, j+k\}]$ are positive since $A$ is $TP_{k+2}$. So $\det D_k(A)[\{i,i+1\},\{j,j+1\}]>0$, and hence by Fekete's Theorem \ref{fekete}, we have that $D_k(A)$ is $TP_2$.
    For the second claim, assume $A$ is $TP_{k+3}$. Then we know $D_k(A)$ is $TP_2$. Consider an arbitrary $3 $-by-$ 3$ contiguous minor of $D_k(A)$, given by $\theta_{i,j} := \det D_k(A)[\{i,i+1,i+2\},\{j,j+1,j+2\}]$. Applying Sylvester's identity (Theorem \ref{thm:sylvester_identity}) and equation (\ref{2x2eqn}) we have
    \begin{align*}
        \theta_{i,j} = \frac{M_{i,j}^{(k+1)}M_{i+1,j+1}^{(k-1)}
    M_{i+1,j+1}^{(k+1)}M_{i+2,j+2}^{(k-1)}
    -M_{i,j+1}^{(k+1)}M_{i+1,j+2}^{(k-1 )}M_{i+1,j}^{(k+1)}M_{i+2,j+1}^{(k-1 )}}{M_{i+1,j+1}^{(k)}}. 
   \end{align*}
   Observe that 
 $M_{i,j}^{(k+1)}M_{i+1,j+1}^{(k+1)}-M_{i,j+1}^{(k+1)}M_{i+1,j}^{(k+1)} =  M_{i,j}^{(k+2)}M_{i+1,j+1}^{(k)}$, and since $A$ is assumed to be $TP_{k+3}$ it follows that $M_{i,j}^{(k+2)}M_{i+1,j+1}^{(k)}>0$ which implies 
$M_{i,j}^{(k+1)}M_{i+1,j+1}^{(k+1)}>M_{i,j+1}^{(k+1)}M_{i+1,j}^{(k+1)}$. Similarly, we have
$M_{i+1,j+1}^{(k-1)}M_{i+2,j+2}^{(k-1)}>M_{i+1,j+2}^{(k-1)}M_{i+2,j+1}^{(k-1)}>0$. Combining these two inequalities implies
\[M_{i,j}^{(k+1)}M_{i+1,j+1}^{(k-1)}
    M_{i+1,j+1}^{(k+1)}M_{i+2,j+2}^{(k-1)}
    -M_{i,j+1}^{(k+1)}M_{i+1,j+2}^{(k-1 )}M_{i+1,j}^{(k+1)}M_{i+2,j+1}^{(k-1 )} >0,\]
    which proves that 
    $\theta_{i,j} = \det D_k(A)[\{i,i+1,i+2\},\{j,j+1,j+2\}]>0.$ Hence by Theorem \ref{fekete}, $D_k(A)$ is $TP_3$. This completes the proof.  
\end{proof}

We provide an example that shows that in general the result of Theorem \ref{main_thm_2} cannot be improved to $TP_4$
\begin{example}\label{eg_for_Thm_B}
    Let us consider the following matrix. 
    \begin{align*}
        A = \begin{bmatrix}
1 & 18 & 192 & 924 & 2332 & 420 \\
32 & 577 & 6161 & 29692 & 75052 & 13524 \\
425 & 7682 & 82145 & 396687 & 1004887 & 181209 \\
2412 & 43807 & 469784 & 2277800 & 5795144 & 1046584 \\
3080 & 56720 & 613350 & 3009027 & 7751484 & 1406076 \\
1440 & 27360 & 301320 & 1515996 & 4007487 & 733594
\end{bmatrix}
    \end{align*}
 Then 
 \begin{align*}
     D_1(A) = \begin{bmatrix}
    1 & 114 & 8100 & 106304 & 16128 \\
599 & 68863 & 4939267 & 64952080 & 10006080 \\
88991 & 10354673 & 752675392 & 9926679328 & 1566406912 \\
1883080 & 222874970 & 16504110168 & 218585490312 & 35833764288 \\
2592000 & 309614400 & 23156130960 & 307417847085 & 51610862484
     \end{bmatrix}
 \end{align*}
One can check that $A$ is $TP_6$ but $D_1(A)$ is not $TP_4$ (the leading principal minor of order $4$ is negative). 
\end{example}

The following theorem, which can be found in \cite{pinkus}, which may not be the original source, provides a useful criteria to determine the total positivity for a Hankel matrix. 
\begin{theorem} \cite[Theorem 4.4]{pinkus}
\label{criteria_of_TP_Hankel}
    Let $A$ be the $(n+1)$-by-$ (n+1)$ Hankel matrix based on the sequence of scalars $a_0,a_1,\ldots,a_{2n}$, and let $A' := A[\{1,2,\ldots, n\},\{2,3,\ldots,n+1\}]$. Then $A$ is $TP$ if and only if both $A$ and $A'$ are positive definite matrices.  
\end{theorem}
We make use of the above result to prove Theorem \ref{main_thm_Hankel}.
\begin{proof}[Proof of Theorem \ref{main_thm_Hankel}] 
Let $A$ be a $TP$ Hankel matrix. Then $A'$ is also a $TP$ Hankel matrix (we use the same notation as defined in Theorem \ref{criteria_of_TP_Hankel}). Since $D_{n}(A)$ is a scalar equal to $\det A$, the statement follows for $k=n$. Next, we assume $1\leq k \leq n-1$. It is easy to check that $D_k(A)$ is also a Hankel matrix and $D_k(A)' = D_k(A')$. Therefore, by Theorem \ref{criteria_of_TP_Hankel} it is enough to prove that $D_k(A)$ is a positive definite matrix as $A$ is an arbitrary $TP$ Hankel matrix. Now, since the matrix $A$ is symmetric, it follows that the compound matrix $C_{k+1}(A)$ is also a symmetric matrix. Furthermore, any eigenvalue of $C_{k+1}(A)$ is a product of some subset of $k+1$ eigenvalues of $A$, and consequently, is positive. That is $C_{k+1}(A)$ is a positive definite matrix. Since $D_k(A)$ is a principal submatrix of $C_{k+1}(A)$ so $D_k(A)$ is also a positive definite matrix. Hence, the proof is complete.
\end{proof}

Finally, we prove our last main theorem.  

\begin{proof}[Proof of Theorem \ref{main_thm_3}]
Clearly $A$ is $TP_{k+1}$ by Fekete's Theorem \ref{fekete}. So we prove A is $TP_{k+2}$. Suppose $D_k(A)$ is $TP_2$. Consider an arbitrary contiguous minor of order $k+2$ of $A$ given by $M_{i,j}^{(k+1)}$. An application of Sylvester's identity (Theorem \ref{thm:sylvester_identity}) reveals,
\[ M_{i,j}^{(k+1)} = \frac{M_{i,j}^{(k)}M_{i+1,j+1}^{(k)}-M_{i,j+1}^{(k)}M_{i+1,j}^{(k)}}{M_{i+1,j+1}^{(k-1)}}.\]
Upon inspection we observe that the numerator above is equal to the $2$-by-$ 2$ contiguous minor of $D_k(A)$ given by $\det D_{k}(A) [\{i,i+1\},\{j,j+1\}]$. All such minors are assumed to be positive since $D_k(A)$ is $TP_2$. Further, since $A$ is $TP_k$ all minors of order $k$ are positive; so in particular, $M_{i+1,j+1}^{(k-1)}>0$. Hence $M_{i,j}^{(k+1)}>0$. Thus, all contiguous minors of order $k+2$ of $A$ are positive, so by Theorem \ref{fekete}, $A$ is $TP_{k+2}$.
\end{proof}

\begin{remark}
Looking back at Theorem \ref{main_thm_3} and comparing it with the implications of Theorem \ref{main_thm_2} we can ask if, in general, $A$ is $TP_k$ and $D_k(A)$ is $TP_{3}$, then $A$ is $TP_{k+3}$? Unfortunately, this is not the case in general. Indeed, there exists a $6 $-by-$ 6$ matrix $B$ that is $TP_3$ and satisfies $D_3(B)$ is TP (or $TP_3$ as $D_3(B)$ is $3 $-by-$ 3)$, but $B$ is not $TP_6$.

Suppose $A$ is a $6 $-by-$ 6$ TP matrix, and set $k=3$. Then using \cite[Lemma 9.5.2]{fallat2022totally}, matrix $B=A-tE_{1,1}$ is a $TP_5$ singular matrix, when
$t=\det(A)/\det(A(\{1\})$ and $E_{1,1}$ is the standard basis matrix whose only nonzero entry is $1$ in the $(1,1)$ position. Applying Theorem \ref{main_thm_2}, we know that $D_3(B)$ is $TP_2$. Furthermore, following the argument similar to the proof of Theorem \ref{main_thm_2} we have
\begin{eqnarray*}
    \det D_3(B) & = & 
\frac{M_{1,1}^{(4)}M_{2,2}^{(2)}
    M_{2,2}^{(4)}M_{3,3}^{(2)}
    -M_{1,2}^{(4)}M_{2,3}^{(2)}
    M_{2,1}^{(4)}M_{3,2}^{(2)}}{M_{2,2}^{(3)}} \\
    & = & 
  \frac{M_{1,1}^{(4)}M_{2,2}^{(4)}\left( M_{2,2}^{(2)}
    M_{3,3}^{(2)}
    -M_{2,3}^{(2)}M_{3,2}^{(2)}\right)}{M_{2,2}^{(3)}} >0. 
\end{eqnarray*}
The second equality follows since 
\[ 0 = \det(B) = \frac{M_{1,1}^{(4)}M_{2,2}^{(4)} - M_{1,2}^{(4)}M_{2,1}^{(4)}}{M_{2,2}^{(3)}},\] and
$M_{2,2}^{(3)}>0$, since $B$ is $TP_5$, while $M_{2,2}^{(2)}
    M_{3,3}^{(2)} >M_{2,3}^{(2)}M_{3,2}^{(2)}$ since $D_3(B)$ is $TP_2$. Thus $D_3(B)$ is $TP_3$. However $B$ is not $TP_{k+3}$, and $B$ is not $TP_6$. 
\end{remark}

\vspace{.5cm}

\subsection*{Acknowledgements}
The workshop “Theory and applications of total positivity", was held at the
American Institute of Mathematics (AIM), from July 24-28, 2023. The authors thank AIM and NSF for their support. In addition, we also thank Sunita Chepuri, Juergen Garloff, Olga Katkova, Daniel Soskin, and Anna Vishnyakova for the thoughtful discussions that occurred during this AIM workshop.

S.M. Fallat was supported in part by an NSERC Discovery Research Grant, Application No.: RGPIN-2019-03934. The work of the PIMS Postdoctoral Fellow H. Gupta leading to this publication was supported in part by the Pacific Institute for the Mathematical Sciences

\vspace{.5cm}


\end{document}